\documentclass[10pt]{amsart}
\usepackage{amsthm,amssymb,latexsym,amsmath,color,mathrsfs,caption}
\usepackage[all]{xy}
\usepackage{mathrsfs,multirow,enumitem,mathdots,marginnote}
\usepackage{mathabx}
\usepackage{dsfont}
\usepackage[utf8]{inputenc}
\usepackage{eucal}
\usepackage[dvipsnames]{xcolor}

\usepackage[left=2.7cm,right=2.7cm,top=2.7cm,bottom=2.7cm]{geometry}
\usepackage{hyperref}

\newcommand{\VV}{{\mathds V}}

\newcommand{\kk}{\mathds k}

\newcommand{\cala}{{\mathcal A}}
\newcommand{\calb}{{\mathcal B}}
\newcommand{\calc}{{\mathcal C}}

\newcommand{\cale}{{\mathcal E}}
\newcommand{\calf}{{\mathcal F}}
\newcommand{\calg}{{\mathcal G}}

\newcommand{\call}{{\mathcal L}}
\newcommand{\calm}{{\mathcal M}}
\newcommand{\caln}{{\mathcal N}}
\newcommand{\calo}{{\mathcal O}}
\newcommand{\calp}{{\mathcal P}}
\newcommand{\calq}{{\mathcal Q}}
\newcommand{\calr}{{\mathcal R}}
\newcommand{\cals}{{\mathcal S}}
\newcommand{\calt}{{\mathcal T}}
\newcommand{\calu}{{\mathcal U}}
\newcommand{\calv}{{\mathcal V}}

\newcommand{\ox}{{\mathcal O}_X}

\newcommand{\opn}{{\mathcal O}_{\mathds{P}^n}}
\newcommand{\pn}{\mathds{P}^n}
\newcommand{\p}[1]{{\mathds{P}^{#1}}}
\newcommand{\op}[1]{{\mathcal O}_{\mathds{P}^{#1}}}

\def\tf{{\mathrm{tf}}}

\newcommand{\sing}{\operatorname{Sing}}
\newcommand{\supp}{\operatorname{Supp}}

\newcommand{\inext}{{\mathcal E}{\it xt}}

\newcommand{\Ext}{\operatorname{Ext}}

\DeclareMathOperator{\coker}{coker}

\DeclareMathOperator{\im}{im}
\DeclareMathOperator{\dv}{div}

\DeclareMathOperator{\codim}{{codim}}
\DeclareMathOperator{\rk}{{rk}}

\DeclareMathOperator{\Pic}{{Pic}}
\newcommand{\lra}{\longrightarrow}
\newcommand{\into}{\hookrightarrow}
\newcommand{\onto}{\twoheadrightarrow}
\newcommand{\Proj}{\mathrm{Proj}}

\newtheorem{theorem}{Theorem}

\newtheorem{proposition}[theorem]{Proposition}

\newtheorem{lemma}[theorem]{Lemma}

\theoremstyle{definition}

\newtheorem{example}[theorem]{Example}

\title{A generalized Saito freeness criterion}

\author[D. Faenzi]{Daniele Faenzi}
\address{Daniele Faenzi.
  Institut de Math{\'e}matiques de Bourgogne, UMR 5584, CNRS \&
  Universit{\'e} de Bourgogne, 9 Avenue Alain
  Savary, BP 47870, 21000 Dijon, France}
\email{daniele.faenzi@u-bourgogne.fr}

\author[M. Jardim]{Marcos Jardim}
\address{Marcos Jardim. Universidade Estadual de Campinas (UNICAMP) \\ Instituto de Matemática, Estatística e Computação Científica (IMECC) \\ Departamento de Matem\'atica \\
Rua S\'ergio Buarque de Holanda, 651\\ 13083-859 Campinas-SP, Brazil}
\email{jardim@unicamp.br}

\author[J. Vallès]{Jean Vallès}
\address{Jean Vall\`es. Universit\'e de Pau et des Pays de l'Adour,
  LMAP-UMR CNRS 5142, 
  Avenue de l'Universit\'e - BP 1155 -
  64013 Pau Cedex, France}
 \email{jean.valles@univ-pau.fr}
\date{\today}

\thanks{
D.F. partially supported by FanoHK ANR-20-CE40-0023, SupToPhAG/EIPHI
ANR-17-EURE-0002, Région Bourgogne-Franche-Comté, JSPS fellowship S24043. M.J. is supported by the CNPQ grant number 305601/2022-9, the FAPESP-ANR project number 2018/21391-1, and the FAPESP Thematic Project number 2021/04065-6. All authors partially supported by Bridges ANR-21-CE40-0017 and CAPES/COFECUB project \emph{Moduli spaces in algebraic geometry and applications}, Capes reference number 88887.191919/2018-00. 
}



\begin{document}
\sloppy

\begin{abstract}
We establish generalizations of Saito's criterion for the freeness of divisors in projective spaces that apply both to sequences of several homogeneous polynomials and to divisors on other complete varieties. As an application, the new criterion is applied to several examples, including sequences whose polynomials depend on disjoint sets of variables, some sequences that are equivariant for the action of a linear group, blow-ups of divisors, and certain sequences of polynomials in positive characteristics.
\end{abstract}

\maketitle

\section{Introduction}

The sheaf $\calt_{X}\langle D \rangle$ of vector fields tangent to some reduced divisor $D$ in a smooth $n$-dimensional complex variety $X$, and its dual sheaf of differential 1-forms with poles of order at most one along $D$ are classical objects studied for decades.
K. Saito in \cite{saito:logarithmic} made the important observation that these sheaves can be locally free
even for badly singular divisors and gave a simple and efficient criterion to check this property, stating that the determinant of the coefficient matrix of a $n$ logarithmic derivations along $D$ should be a defining equation of $D$.  It can be applied when $X$ is the affine or projective space to check that the module of logarithmic derivations along a divisor $D$ is free, or equivalently that the sheaf $\calt_{X}\langle D \rangle$ is a direct sum of line bundles, in which case we say that $D$ is \textit{free}.

More recently, the authors introduced in \cite{faenzi-jardim-valles} a generalization of Saito's theory from divisors to subvarieties of $\pn$ with codimension higher than $1$. One goal of the present paper is to provide a generalization of Saito's criterion that applies to our new theory. To be precise, let $\sigma=(f_1,\dots,f_k)$ be a  sequence of algebraically independent homogeneous polynomials in $\kk[x_0,\dots,x_n]$ for some field $\kk$. If ${\rm char}(\kk) > 0$ we assume that the Jacobian matrix has rank $k$, which happens for instance if ${\rm char}(\kk)>\max_{1 \le i \le k}(\deg(f_i))^k$, cf. \cite{anurag-saxena-sinhababu}. In \cite{faenzi-jardim-valles} the authors considered the Jacobian $k\times(n+1)$ matrix $\nabla(\sigma)$, whose $i^{\rm th}$ line consists of the partial derivatives of the polynomial $f_i$, as a morphism
$$ \nabla(\sigma): \opn^{\oplus n+1} \longrightarrow \bigoplus_{i=1}^k \opn(d_i).$$

The main character considered in \cite{faenzi-jardim-valles} is the sheaf $\calt_\sigma:=\ker(\nabla(\sigma))$, called the \textit{logarithmic tangent sheaf} associated with the sequence $\sigma$. If $\sigma$ is a regular sequence of base locus $Z$, then $\calt_\sigma$ is a subsheaf of $\calt_{\p n}\langle Z\rangle$, and the interplay between these sheaves is described by a basic exact sequence, see \cite[Lemma 2.4]{faenzi-jardim-valles}.
While $\calt_{\p n}\langle Z\rangle$ cannot be reflexive if $k \ge 2$, $\calt_\sigma$ is always reflexive, possibly locally free or even isomorphic to a direct sum of line bundles, in which case we say that $\sigma$ is \textit{locally-free} and \textit{free}, respectively.

The first main result of this paper gives an effective criterion to check whether $\calt_\sigma$ is isomorphic to an arbitrary reflexive sheaf $\cale$, rather than just the direct sum of line bundles. To formulate it, given a map $\cale \to \calt_\sigma \subset \calo_{\p n}^{\oplus n+1}$, we complete it via the Euler derivation to a map $\theta : \cale \oplus \calo_{\p n}(-1) \to \calo_{\p n}^{\oplus n+1}$.

\begin{theorem} \label{thm:1}
We have  $\calt_\sigma \simeq \cale$ if there is $h\in \kk^\times$ such that
$$ \gcd\Big(\bigwedge\theta\Big) \cdot \gcd\Big(\bigwedge \nabla(\sigma)\Big) =h\cdot\gcd\Big(\bigwedge \sigma\Big). $$
\end{theorem}

Here $\bigwedge f$ refers to the maximal minors of a given map $f$. The result is formulated in greater generality in the paper, see Theorem \ref{thm:super-saito}. It boils down to the Saito criterion on $\p n$ when $k=1$, see Lemma \ref{usual saito}. 
We apply it to sequences $\sigma$ whose elements depend on disjoint sets of variables, see Examples \ref{eg1} and \ref{eg2}. It also applies to some sequences $\sigma$ which are equivariant for the action of a linear group, cf. Example \ref{eg3}. A version of this criterion valid for derivation modules over the polynomial ring was given in \cite{affine-saito}.

Anyway, a major limitation of the global version of the Saito criterion is that it only applies to divisors in affine or projective spaces. Our Theorem \ref{thm:super-saito} mentioned above is actually derived from a general result in sheaf theory, namely Theorem \ref{thm:main} below. This allows for instance to prove the following result (see Theorem \ref{thm:blowups} for a more general version).

\begin{theorem} \label{thm:2}
Let $X$ be a smooth complete variety, $C$ is a reduced divisor of $Y$, $p_1,\ldots,p_k \in Y \setminus C$. Let $X$ be the blow-up of $Y$ at $p_1,\ldots,p_k$, $D$ the strict transform of $C$ and $E$ the exceptional divisor.
Then    $\calt_X\langle D\rangle \simeq \pi^*(\calt_Y\langle C\rangle)(-E)$, so $D$ is free whenever $C$ is free.
\end{theorem}

Finally, our second technical result (Theorem \ref{thm:alternative} below) is used to show that if $\sigma=(f_1,\dots,f_{n-1})$ is a sequence of algebraically independent homogeneous polynomials in $\kk[x_0,\dots,x_n]$ such that ${\rm char}(\kk)$ divides $\deg(f_i)$ for each $i=0,\dots,n-1$ then $\sigma$ is free, cf. Theorem \ref{thm:positive} in Section \ref{sec:char>0}.

The paper is logically ordered differently than the narrative presented so far. Indeed, we first carry out some basic material, with two general statements in sheaf theory (Theorem \ref{thm:main} and Theorem \ref{thm:alternative}) that will form the basis of our applications to freeness criteria. Section \ref{sec:appl} is dedicated to the main applications of Theorems \ref{thm:main} and \ref{thm:alternative}, up to some extra work. More precisely, the proof of Theorem \ref{thm:1} is given in Section \ref{sec:free}, where we also explain how it implies the usual Saito criterion for divisors in projective spaces and give examples. Theorem \ref{thm:2} is proved in Section \ref{sec:blowups}. We complete the paper with Section \ref{sec:char>0} by establishing our last claim regarding the freeness of certain sequences of algebraically independent homogeneous polynomials in positive characteristics.

\section{Isomorphy criteria for reflexive sheaves} \label{sec:tech}

Let $X$ be a smooth connected complete variety over a field $\kk$, and consider coherent, torsion-free sheaves $\cala$ and $\calb$ on $X$ of rank $a$ and $b$ respectively. Let $\alpha:\cala\to\calb$ be a morphism; we say that $\alpha$ has maximal rank if it is either injective, when $a\le b$, or generically surjective, when $a\ge b$. In the last case, note that $\coker(\alpha)$ is a torsion sheaf. We introduce the following notation:
\begin{enumerate}
\item the \textit{divisor} of a
morphism $\alpha$ is defined as
\[
\dv(\alpha) := c_1(\tau(\coker(\alpha))) \in \Pic(X),
\]
where $\tau(\calq)$ stands for the torsion part of a coherent sheaf $\calq$.
\item the \textit{degeneracy locus} of a monomorphism $\alpha$ is defined as
$$ D(\alpha):= \supp \big( \inext^1(\coker(\alpha),\ox) \big), $$
regarded as a closed subset of $X$.
\end{enumerate}
Note that $\dv(\alpha)=0$ if and only if $\codim\big(D(\alpha)\big)\ge2$. The nomenclature for $D(\alpha)$ comes from the following observation: when $\cala$ and $\calb$ are locally free, then 
$$ D(\alpha) = \{ p\in X ~|~ \alpha(p) ~ \textrm{is not injective} \} $$
where $\alpha(p)$ denotes the fiber map $\cala(p)\to\calb(p)$, thus coinciding with the usual notion for the degeneration locus of a monomorphism.

The present section aims to prove the following two technical results that will be useful in the main body of the article. Given divisors $D_1,D_2 \in \Pic(X)$, we write $D_1 \le D_2$ to express that $D_1-D_2$ is not effective. We assume from now on that $\alpha$ is generically surjective.

\begin{theorem} \label{thm:main}
Let $\calf$ be a reflexive sheaf and consider a monomorphism $\theta:\calf\to\cala$; assume that $a>b$. If $\rk(\ker(\alpha\circ\theta))=a-b$ and $\dv(\theta) + \dv(\alpha) \le \dv(\alpha\circ\theta)$,
then $\ker(\alpha)\simeq\ker(\alpha\circ\theta)$.
\end{theorem}

In the previous setting, put $L:=\dv(\alpha)+c_1(\cala)-c_1(\calf)-c_1(\calb)$ and let $\call$ be the associated line bundle. We then have the following result.

\begin{theorem} \label{thm:alternative}
Assume that $\cala$ is locally free. 
Let $\calf$ be a locally free sheaf with $\rk(\calf)=a-b-1>0$ and assume that $h^1(\calf(-L))=0$.
If $\theta:\calf\to\cala$ is a monomorphism such that $\alpha\circ\theta=0$ and $\codim\big(D(\theta)\big)\ge3$, then $\ker(\alpha)\simeq\calf\oplus\call$.
\end{theorem}

The proofs require a few general, technical lemmas, which we will discuss in the next three subsections.

\subsection{An isomorphy criterion} \label{sec:iso}
We start with the following lemma, keeping the notation  and conventions provided at the beginning of this section

\begin{lemma}\label{lem:sheaves}
Assume that $a=b$ and that $\cala$ is reflexive. If $c_1(\calb) \le c_1(\cala)$. Then $\alpha$ is an isomorphism.
\end{lemma}

\begin{proof}
Since $\rk(\cala)=\rk(\calb)$ and $\alpha$ is injective, $\coker(\alpha)$ is a torsion sheaf, which we assume to be nonzero by contradiction. The divisorial locus of the support of $\coker(\alpha)$ is of class  $c_1(\calb)-c_1(\cala)$.
As this divisor class is not effective by assumption, we get $c:=\codim(\coker(\alpha))\ge2$. It follows that $\inext^1(\coker(\alpha),\ox)=0$, so dualizing the exact sequence:
\begin{equation}\label{eq:sqc1}
0 \longrightarrow \cala \stackrel{\alpha}{\longrightarrow} \calb \longrightarrow \coker(\alpha) \longrightarrow 0,
\end{equation}
we get $\cala^\vee\simeq\calb^\vee$, and 
exact sequences
$$ \inext^{c-1}(\cala,\ox) \lra \inext^c(\coker(\alpha),\ox) \lra \inext^c(\calb,\ox). $$
The middle term has codimension equal to $c$ while,  according to \cite[Proposition 1.10]{huybrechts-lehn:moduli},
$\codim(\inext^c(\calb,\ox))\ge c+1$ (because $\calb$ is torsion free) and $\codim(\inext^{c-1}(\cala,\ox))\ge c+1$ (because $\cala$ is reflexive). This is a contradiction, so we must have $\coker(\alpha)=0$, therefore $\alpha$ is an isomorphism.
\end{proof}

The proof of our next lemma has a similar spirit.

\begin{lemma}\label{lem:reflexive}
Assume that $a<b$, $\cala$ is locally free and $\calb$ is  reflexive. If $\codim\big(D(\alpha)\big)\ge3$, then $\coker(\alpha)$ is a reflexive sheaf.
\end{lemma}
\begin{proof}
Under the current assumptions, $\alpha$ is a monomorphism, so we obtain a short exact sequence as in display \eqref{eq:sqc1}; dualizing it, 
we obtain the exact sequence
$$ 0 \lra \coker(\alpha)^\vee \lra \calb^\vee \stackrel{\alpha^\vee}{\lra} \cala^\vee \lra \inext^1(\coker(\alpha),\ox) \lra \inext^1(\calb,\ox) \lra 0, $$
and isomorphisms $\inext^p(\coker(\alpha),\ox) \simeq \inext^{p}(\calb,\ox)$ for each $p\ge2$. By hypothesis,
$\codim(\inext^1(\coker(\alpha),\ox))\ge3$; in addition, since $\calb$ is reflexive, \cite[Lemma 1.1.10]{huybrechts-lehn:moduli}
guarantees that $\codim(\inext^p(\coker(\alpha),\ox))\ge p+2$, thus $\coker(\alpha)$ is a reflexive sheaf.
\end{proof}

\subsection{Torsion and torsion-free parts of cokernel sheaves} \label{sec:ses}

Another useful general fact is the following. 
Let $\calu$ be a coherent sheaf on $X$.
Recall that $\tau(\calu)$ denotes the torsion part of $\calu$, namely the maximal torsion subsheaf of $\calf$, so that the quotient $\calu^{\rm tf}:=\calu/\tau(\calu)$ is torsion free.

Given a morphism $\beta:\calu\to\calv$ between coherent sheaves, we observe that there are induced morphisms $\tau(\beta):\tau(\calu)\to\tau(\calv)$ and $\beta^{\rm tf}:\calu^{\rm tf}\to\calv^{\rm tf}$; in addition, both $\tau(\beta)$ and $\beta^{\rm tf}$ are monomorphisms when $\beta$ is. Indeed, the composition
$$ \tau(\calu) \into \calu \stackrel{\beta}{\to} \calv \onto \calv^{\rm tf} $$
vanishes, since it gives a morphism from a torsion sheaf to a torsion-free one. It follows that $\beta(\tau(\calu))\subseteq\tau(\calv)$, thus we get a commutative diagram
\begin{equation}\label{diag-torsion}
\begin{split} \xymatrix@R-2ex{ 
0 \ar[d] & 0 \ar[d]  \\
\tau(\calu)\ar[d] \ar@{.>}[r]^{\tau(\beta)}& \tau(\calv)\ar[d] \\
\calu \ar[r]^{\beta}\ar[d] & \calv \ar[d] \\
\calu^{\rm tf} \ar@{.>}[r]^{\beta^{\rm tf}}\ar[d] & \calv^{\rm tf} \ar[d] \\
0 & 0 
} \end{split}
\end{equation}
Finally, if $\beta$ is injective, $\tau(\beta)$ is also injective. It also follows that $\ker(\beta^{\rm tf})$ is a subsheaf of $\coker(\tau(\beta))$; however, the latter is a torsion sheaf, while the former is a torsion-free sheaf, thus, in fact, $\ker(\beta^{\rm tf})=0$. We summarize our conclusions in the following statement.

\begin{lemma}\label{lem:torsion}
Let $\beta:\calu\to\calv$ be an injective morphism between coherent sheaves on a scheme $X$. Then there is an exact sequence
$$ 0\longrightarrow \coker(\tau(\beta)) \longrightarrow \coker(\beta) \longrightarrow \coker(\beta^{\rm tf}) \longrightarrow 0 $$
\end{lemma}
\begin{proof}
As we noted above, if $\beta$ is injective, then so are $\tau(\beta)$ and $\beta^{\rm tf}$. The exact sequence is an immediate consequence of the snake lemma applied to the diagram in display \eqref{diag-torsion}.
\end{proof}

\subsection{Proof of Theorem \ref{thm:main}} \label{sec:proof}
The morphisms $\alpha$ and $\theta$ induce the following commutative diagram
\begin{equation} \label{thetas}
\begin{split} \xymatrix@-1ex{
& 0 \ar[d] & 0 \ar[d] & 0 \ar[d] & \\
0 \ar[r] & \ker(\alpha\circ\theta) \ar[r] \ar^-{\theta_0}[d] & \calf \ar[r] \ar^-{\theta}[d]& \im(\alpha\circ\theta) \ar[r] \ar^-{\theta_2}[d] & 0\\
0 \ar[r] & \ker(\alpha) \ar[r] & \cala \ar_-{\alpha}[r] & \calb  } 
\end{split}\end{equation}
where $\theta_0$ and $\theta_2$ are the naturally induced morphisms. Clearly, $\theta_2$ factors through $\im(\alpha)$; denote by $\bar\theta_2$ the induced morphism $\im(\alpha\circ\theta) \to \im(\alpha)$. We set 
\begin{align*}
& \calf_0:=\ker(\alpha\circ\theta), && \cale_0:=\ker(\alpha), && \calf_2:=\im(\alpha\circ\theta), \\
&\calc_0:=\coker(\theta_0), && \calc_1:=\coker(\theta), &&\calc_2:=\coker(\bar \theta_2)
\end{align*}
we write the commutative exact diagram
\[
\xymatrix@-2ex{
& 0 \ar[d] & 0 \ar[d] & 0 \ar[d] \\
0 \ar[r] & \calf_0 \ar[r] \ar^-{\theta_0}[d] & \calf \ar[r] \ar^-{\theta}[d]& \calf_2 \ar[r] \ar^-{\bar \theta_2}[d] & 0\\
0 \ar[r] & \cale_0 \ar[r] \ar[d]& \cala \ar_-{\alpha}[r] \ar[d]& \im(\alpha) \ar[d]\ar[r] & 0 \\ 
0 \ar[r] & \calc_0 \ar^-{\gamma}[r]  \ar[d] &  \calc_1 \ar[r] \ar[d]&  \calc_2 \ar[r] \ar[d]&  0\\
& 0 & 0 & 0
} 
\]
We apply Lemma \ref{lem:torsion} to the morphism $\gamma:\calc_0\to\calc_1$ induced by the diagram above. Since $\rk(\ker(\alpha\circ\theta))=\rk(\ker(\alpha))$, the sheaf $\calc_0$ is a torsion sheaf, so we get an exact sequence:
\[
0 \to \calc_0 \to \tau(\calc_1) \to \calc_2 \to \calc_1^\tf \to 0,
\]
where the image of the middle map is $\coker(\tau(\gamma))$.
Such an image is a torsion sheaf, so it must lie in $\tau(\calc_2)$ and we get a quotient sheaf $\calp$ fitting into:
\[
0 \to \coker(\tau(\gamma)) \to \tau(\calc_2) \to \calp \to 0, \qquad 
0 \to \calp \to \calc_1^\tf \to \calc_2^\tf \to 0.
\]
Now $\calp$ is torsion by the first sequence and torsion-free by the second one so that $\calp=0$ and $\coker(\tau(\alpha)) \simeq \tau(\calc_2)$. We get:
\begin{equation} \label{first equation c1}
    c_1(\cale_0) - c_1(\calf_0)=c_1(\calc_0) = c_1(\tau(\calc_1)) - c_1(\tau(\calc_2)) = \dv(\theta) - \dv(\bar \theta_2).
\end{equation}

Next, we consider the exact commutative diagram
\[
\xymatrix@-2ex{
&0 \ar[d] & 0 \ar[d] \\
&\calf_2 \ar@{=}[r] \ar[d] & \calf_2 \ar[d]^{\theta_2} \\
0 \ar[r] & \im(\alpha) \ar[d] \ar[r] & \calb \ar[d] \ar[r] & \calq \ar[r] \ar@{=}[d] &  0 \\
0 \ar[r] & \calc_2 \ar^-{\beta}[r] \ar[d] & \calr \ar[r] \ar[d] & \calq \ar[r] & 0\\
& 0 & 0
}
\]
Here, $\beta$ is defined by the diagram and we set
\[
\calr := \coker(\theta_2)\simeq\coker(\alpha\circ\theta) \quad {\rm and} \quad \calq := \coker(\alpha);
\]
the isomorphism above comes from the rightmost column of the diagram in display \ref{thetas}. We apply again Lemma \ref{lem:torsion}, this time to $\beta$, and get an exact sequence
\[
0 \to \tau(\calc_2) \to \tau(\calr) \to \calq \to \coker(\beta^\tf) \to 0,
\]
the image of the middle map being $\coker(\tau(\beta))$.
Since this image is a torsion sheaf, it must be contained in $\tau(\calq)$, the quotient being a torsion sheaf that we denote by $\cals$, fitting into:
\[
0 \to \coker(\tau(\beta)) \to \tau(\calq) \to \cals \to 0.
\]
Since $\cals$ is a torsion sheaf, $c_1(\cals)$ is effective and $-c_1(\cals)$ is not.
We get :
\begin{equation} \label{second equation c1}
    \dv(\bar \theta_2) = c_1(\tau(\calc_2)) = c_1(\tau(\calr)) - c_1(\tau(\calq)) + c_1(\cals) = \dv(\theta_2) -\dv(\alpha) + c_1(\cals).
\end{equation}

Plugging into \eqref{first equation c1}, using the assumption $\dv(\theta) + \dv(\alpha) - \dv(\alpha\circ\theta) \le 0$ and the fact that $\dv(\theta_2)=\dv(\alpha\circ\theta)$ while $-c_1(\cals)$ is not effective, we obtain:
\[
c_1(\cale_0) - c_1(\calf_0) = \dv(\theta) + \dv(\alpha) - \dv(\alpha\circ\theta)- c_1(\cals) \le 0
\]
Therefore, $\theta_0$ is an isomorphism by Lemma \ref{lem:sheaves}.

This concludes the proof of Theorem \ref{thm:main}. \hfill $\Box$

\subsection{Proof of Theorem \ref{thm:alternative}} \label{sec:proof2}

Since $\alpha\circ\theta=0$, we obtain an induced monomorphism $\theta_0:\calf\to\cale_0$, where $\cale_0=\ker(\alpha)$, and the following exact commutative diagram
\begin{equation} 
\begin{split} \xymatrix@-1ex{
& 0 \ar[d] & 0 \ar[d] &  & \\
0 \ar[r] & \calf \ar[r] \ar^-{\theta_0}[d] & \calf \ar^-{\theta}[d] &  &  \\
0 \ar[r] & \cale_0 \ar[r] \ar[d] & \cala \ar_-{\alpha}[r] \ar[d] & \calb \ar@{=}[d] \\
0 \ar[r] & \call \ar[r] \ar[d] & \coker(\theta) \ar[r] \ar[d] & \calb \\
& 0 & 0 &  &
} 
\end{split}\end{equation}
where $\call:=\coker(\theta_0)$. By Lemma \ref{lem:reflexive}, $\coker(\theta)$ is a reflexive sheaf. It follows that $\call$ is the kernel of a morphism from a reflexive sheaf to a torsion-free sheaf, thus $\call$ is a reflexive sheaf \cite[Lemma 1.1.16]{okonek-schneider-spindler}. Since, in addition, $\rk(\call)=1$, we conclude that $\call$ is a line bundle \cite[Lemma 1.1.15]{okonek-schneider-spindler}. Note that
$$ \call \simeq \det(\calf)^{-1}\otimes\det(\cale_0) \simeq \det(\calf)^{-1}\otimes\det(\cala)\otimes\det(\calb)^{-1}\otimes\dv(\alpha) $$

The leftmost column displays $\cale_0$ as an extension of $\call$ by $\calf$. But $\Ext^1(\call,\calf)=H^1(\calf\otimes\call^{-1})=0$ by hypothesis, so in fact $\cale_0\simeq\calf\oplus\call$. \hfill $\Box$

\subsection{Characterizing the divisor of a morphism} \label{sec:divisor}

Let $\theta:\calf \to \cala$ be a monomorphism and set $r=\rk(\calf)$; assume that $\cala$ is a locally free sheaf such that $\bigwedge^r\cala$ is isomorphic to a direct sum of line bundles. 
The following lemma shows that, under these hypotheses, an equation defining the divisor representing $\dv(\theta)$ can be computed as the greatest common divisor of the entries of the maximal exterior power $\bigwedge\theta$ of $\theta$:
\[
\bigwedge^r \theta : =  \bigwedge^r \calf ^{\vee \vee} \simeq \det(\calf) \longrightarrow \bigwedge^r \cala \simeq \bigoplus_{i=1}^k \calm_{i}, \qquad \calm_i\in\Pic(X), \qquad k:={a\choose r}.\] 

To be more precise, note that $\bigwedge\theta$ is given by a $k$-tuple $g=(g_1,\ldots,g_k)$ with each $g_j$ being a global section of $\calm_j\otimes\det(\calf)^\vee$. A common divisor of $g$ is a global section $h$ of $\call\in\Pic(X)$ such that $g_j=h\circ q_j$ for each $j=1,\dots,k$ and for some global section $q_j\in H^0(\call^\vee\otimes\calm_j\otimes\det(\calf)^\vee)$; in other words, we have the diagram:
$$ \xymatrix{ \ox \ar[r]_{h} \ar@/^2pc/[rr]|-{g_j} & \call \ar[r]_{q_j \hspace{8mm}} & \calm_j\otimes\det(\calf)^\vee } . $$

In particular, if $\cala$ also has rank equal to $r$, then $\bigwedge\theta$ is just a single section of $\det(\cala)\otimes\det(\calf)^\vee$, so $\gcd(\bigwedge\theta)$ is simply $\det(\theta)$.

\begin{lemma} \label{lem:c_1=gcd}
Let $\calf$ be a torsion-free of rank $r$ and $\cala$ be a locally free sheaf such that $\bigwedge^r\cala$ is isomorphic to a direct sum of line bundles.
If $\theta:\calf\to\cala$ is a monomorphism, then:
\begin{enumerate}[label=\roman*)]
\item $\coker(\theta)$ is torsion free if and only if $\gcd(\bigwedge\theta)\in\kk^\times$.
\item More generally, $\dv(\theta)\equiv\VV(\gcd(\bigwedge\theta)))$.
\end{enumerate} 
\end{lemma}
\begin{proof}
Set $\calq:=\coker(\theta)$.
For the first claim, dualize the short exact sequence
$$ 0 \longrightarrow \calf \stackrel{\theta}{\longrightarrow} \cala \longrightarrow \calq \longrightarrow 0 $$
to obtain 
$$ 0 \longrightarrow \calq^\vee \longrightarrow \calf^\vee \stackrel{\theta^\vee}{\longrightarrow} \cala^\vee \longrightarrow \inext^1(\calq,\calo_X) \longrightarrow 0, $$
and $\inext^{c}(\calq,\opn)\simeq\inext^{c+1}(\calf,\opn)$ for $p\ge1$.
In addition, note that
$$ \VV\big(\bigwedge\theta\big)=\supp(\coker(\theta^\vee))=\supp(\inext^1(\calq,\opn)). $$

If $\calq$ is torsion-free, then $\codim\supp(\inext^1(\calq,\calo_X))\ge2$, thus $\VV\big(\bigwedge\theta\big)$ does not contain a hypersurface, thus $\gcd(\bigwedge\theta)\in\kk^\times$. 

Conversely, if $\gcd(\bigwedge\theta)\in\kk^\times$, then $\codim \VV\big(\bigwedge\theta\big)\ge2$, and it follows that $\codim\inext^{c}(\calq,\calo_X)\ge c+1$ for every $c>0$, so $\calq$ is torsion free, see \cite[Proposition 1.10]{huybrechts-lehn:moduli}.

\medskip
Let us now check the second statement.
Let $\calf'$ denote the kernel of the composed epimorphism
$$ \cala \onto \coker(\theta) \onto \coker(\theta)^{\rm tf}; $$
note that it fits into the following short exact sequence
$$ 0 \longrightarrow \calf \longrightarrow \calf' \longrightarrow \tau(\coker(\theta)) \longrightarrow 0, $$
so that
$$ \dv(\theta)=c_1(\tau(\coker(\theta)))=c_1(\calf')-c_1(\calf). $$

Write $\theta':\calf'\into\cala$ the induced inclusion and consider the commutative square:
\begin{equation} \label{diag-torsion1}
\begin{split} \xymatrix@-2ex{ 
\calf\ar[d] \ar[r]^{\theta}& \cala \ar@{=}[d] \\
\calf' \ar[r]^{\theta'} & \cala
} \end{split}
\end{equation}

It yields, after taking maximal exterior powers and double duals, the diagram
\begin{equation} \label{diag-torsion2}
\begin{split} \xymatrix@R-2ex{ 
\bigwedge^r \calf ^{\vee \vee} \ar[d]^-g \ar[r]^-{\bigwedge\theta}& \bigwedge^r \cala \ar@{=}[d] \\
\bigwedge^r (\calf') ^{\vee \vee} \ar[r]^-{\bigwedge\theta'} & \bigwedge^r \cala
} \end{split}
\end{equation}

Note that $g$ is an injective morphism between line bundles, so its cokernel is a torsion sheaf supported on a divisor of class $\dv(\theta)$, as:
\[
c_1\left(\bigwedge^r (\calf') ^{\vee \vee}\right)-c_1\left(\bigwedge^r \calf ^{\vee \vee}\right) = c_1(\calf')-c_1(\calf)=\dv(\theta).
\]

Since $\bigwedge\theta=g\cdot\bigwedge\theta'$ and $\coker(\theta')=\coker(\theta)^{\rm tf}$ is a torsion free sheaf, we have that $\gcd(\bigwedge\theta')\in\kk^\times$ thus $g=\gcd(\bigwedge\theta)$, and the second claim follows. 
\end{proof}

Lemma \ref{lem:c_1=gcd} allows us to check that the converse of Theorem \ref{thm:main} is not generally true.

To see this, take $X=\p2$, $\cala=\op2^{\oplus3}$, $\calb=\op2(1)^{\oplus2}$ let $\alpha$ be the morphism given by the matrix
$$ \alpha = \left( \begin{array}{cccc}
x_1 & x_0 & 0  \\
x_2 & 0 & x_0
\end{array} \right)$$
where $(x_0,x_1,x_2)$ are homogeneous coordinates in $\p2$; the second part of Lemma \ref{lem:c_1=gcd} yields $\dv(\alpha)\cong \VV(x_0)$. In addition, one can check that $\ker(\alpha)\simeq\op2(-1)$ (it must be a reflexive rank 1 sheaf with $c_1=-1$).

Now let $\calf=\op2(-1)^{\oplus2}$ and consider the morphism $\theta:\calf\to\cala$ given by the matix
$$ \theta = \left( \begin{array}{cc}
x_0 & x_0 \\ -x_1 & x_1 \\ -x_2 & x_2
\end{array} \right). $$
It is easy to see that $\ker(\alpha\circ\theta)=\op2(-1)$, so it is isomorphic to $\ker(\alpha)$. However, one can check that $\dv(\theta)\cong\dv(\alpha\circ\theta)\cong\VV(x_0)$ so the inequality in the hypothesis of Theorem \ref{thm:main} is not satisfied.

We conclude this section by showing that a converse to Theorem \ref{thm:main} can be obtained after an additional condition on $\dv(\varphi)$ is assumed.

\begin{proposition} \label{prop:converse}
Let $\calf$ be a reflexive sheaf and consider a monomorphism $\theta:\calf\to\cala$; assume that $a>b$. If the natural monomorphism $\iota:\ker(\alpha)\hookrightarrow\cala$ factors through $\theta$ and $\dv(\alpha)=0$, then $\dv(\theta)=\dv(\alpha\circ\theta)$.
\end{proposition}

\begin{proof}
Fix the following notation 
$$ \cale_0:=\ker(\alpha), \qquad \calf_2:=\im(\alpha\circ\theta), \qquad \calc=\coker(\theta), \qquad \mbox{and} \qquad \calr=\coker(\alpha\circ\theta). $$
The condition of the statements yields the following exact commutative diagram
\begin{equation} \label{thetas2}
\begin{split} \xymatrix@-1ex{
&  & 0 \ar[d] & 0 \ar[d] & \\
0 \ar[r] & \cale_0 \ar[r] \ar@{=}[d] & \calf \ar[r] \ar^-{\theta}[d]& \calf_2 \ar[r] \ar^-{\theta_2}[d] & 0\\
0 \ar[r] & \cale_0 \ar[r] & \cala \ar_-{\alpha}[r] \ar[d] & \calb \ar[d] \\
 &  & \calc \ar_-{\gamma}[r]\ar[d] & \calc_2\ar[d] \\
  &  & 0 & 0} 
\end{split}\end{equation}

Here, the morphisms $\theta_2$ and $\gamma$ are induced by the diagram. In addition, notice that $\ker(\alpha)\simeq\ker(\alpha\circ\theta)$. It follows that $\gamma$ is a monomorphism and $\coker(\gamma)\simeq\coker(\alpha)$.

Applying the arguments of Section \ref{sec:ses}, we obtain a monomorphism $\tau(\gamma):\tau(\calc)\hookrightarrow\tau(\calr)$ whose cokernel is a subsheaf of $\coker(\alpha)$. Since $\dv(\alpha)=0$, $\codim\coker(\alpha)\ge2$, thus $c_1(\tau(\calc))=c_1(\tau(\calr))$, which can be translated into the equality $\dv(\theta)=\dv(\alpha\circ\theta)$.
\end{proof}

\section{Applications to freeness} \label{sec:appl}

In this section, we develop the main applications of our results of the previous section, mostly in the direction of showing that some divisors or more generally some algebraically independent sequences give rise to reflexive, locally free, or free sheaves of logarithmic derivations, one of the most natural generalizations of the idea of free divisor being that the sheaf of logarithmic derivations along a divisor splits as a direct sum of line bundles.

In the next subsection, we spell out the Saito criterion for reduced divisors of a smooth complete variety $X$ and recall how Saito's global criterion for divisor in $\p n$ fits our discussion. Then, in \S \ref{sec:blowups} we look at a particular class of divisor of blow-ups at points.
In \S \ref{sec:free} we apply our method to study the freeness of logarithmic sheaves attached to algebraically independent families on $\p n$.
Finally in \S \ref{sec:char>0} we point out a result on freeness in positive characteristic.

\subsection{Saito criterion for hypersurfaces}
\label{sec:hyper}

We now explain how Theorem \ref{thm:super-saito} gives a Saito criterion for hypersurfaces of a given variety $X$.

Let $X$ be a smooth connected complete variety of dimension $n$ over a field $\kk$. Let $D$ be a geometrically reduced divisor of $X$, defined by an equation $f \in H^0(\calo_X(D))$. We consider the logarithmic tangent sheaf $\calt_X\langle D \rangle$ defined as the kernel of the natural composition:
\[
\tau_D : \calt_X \to \calt_X |_D \to \calo_D(D).
\]

Let $\calf$ be a reflexive sheaf of rank $n$ and consider an injective map $\theta : \calf \to \calt_X.$
\begin{proposition}
Assume that $\tau_D \circ \theta = 0$. Then $\det(\theta) = \lambda f$, with $\lambda \in \kk^\times$ if and only if $\theta$ induces an isomorphism $\calf \simeq \calt_X\langle D \rangle$. This happens if and only if $\dv(\theta) \le D$.
\end{proposition}
\begin{proof}
    Since $\calt_X\langle D \rangle = \ker(\tau_D)$ and $\tau_D \circ \theta=0$, we get an induced map $\theta_0 : \calf \to \calt_X\langle D \rangle$. The map $\theta_0$ is still injective, it is thus of maximal rank, hence by Lemma \ref{lem:sheaves}, $\theta$ induces an isomorphism if and only if $c_1(\calf) \ge c_1(\calt_X\langle D\rangle)=c_1(X)-D $, i.e. if and only if $\dv(\theta)=c_1(X) - c_1(\calf) \le D$. 
    Note that $\det(\theta)$ vanished along $D$, hence $f$ divides 
    $\det(\theta)$ and there is an effective divisor $D$' and $g \in H^0(\calo_X(D'))$ such that $\det(\theta)=fg$. Hence $\dv(\theta)=D+D'$. Then $D'$ is empty if and only if $g$ is a nonzero constant, which happens if and only if $\dv(\theta) \le D$.
\end{proof}
\begin{example}
    Let $X_1,\ldots,X_m$ be smooth complete varieties, and for $1 \le i \le m$, let $D_i$ be a reduced effective divisor in $X_i$.  Set $X=X_1\times \cdots \times X_m$,
    let $p_i : X \to X_i$ be the $i$-th projection,
    put $F_i = p_i^*(D_i)$, for $1 \le i \le m$ and consider
    $F = F_1 \cup \cdots \cup F_m$.
    Taking $\theta$ to be a diagonal map whose blocks are pull-backs of the obvious injection $\calt_{X_i}\langle D_i \rangle \to \calt_{X_i}$ and applying the proposition, we get
    \[
    \calt_X\langle F \rangle = \bigoplus_{1 \le i \le m} \pi^*(\calt_{X_i}\langle D_i \rangle).
    \]
    Hence, if $D_i$ is free for all $1 \le i \le m$, then $F$ is also free. For instance when $\dim(X_i)=1$ for all $1 \le i \le m$:
    \[
    \calt_X\langle F \rangle = \bigoplus_{1 \le i \le m} \pi^*(\omega_{X_i}^\vee(- D_i)).
    \]
\end{example}
\begin{example}
    Let $E$ be a $(-1)$-curve on a smooth complete surface $X$. Blowing down $E$ we get a morphism $\pi$ to a smooth complete surface $Y$. Multiplying by $E$ gives a map $\pi^*(\calt_Y)(-E) \to \pi^*(\calt_Y)$ that factors through $\calt_X$. We get a map $\theta : \pi^*(\calt_Y)(-E) \to \calt_X$ and we compute $\dv(\theta)=E$, hence by the proposition 
    \[
    \calt_X\langle E\rangle \simeq \pi^*(\calt_Y)(-E).
    \]
    Therefore, $E$ is free for instance when $Y$ is an abelian surface or a product of smooth projective curves.
\end{example}

\subsection{Blow-ups} \label{sec:blowups}

Let $C$ be a hypersurface of a smooth complete variety $Y$ and 
\[
p_1,\ldots,p_k \in Y \setminus \sing(C).
\]
Consider the blow-up $\pi : X \to Y$ of $Y$ at $p_1,\ldots,p_k$ and let $E_1,\ldots,E_k$ be the exceptional divisors of $X$ lying over $p_1,\ldots,p_k$ and $E=E_1\cup \cdots \cup E_k$. 
Let $\tilde{C} \subset X$ be the strict transform of $C$  and consider the hypersurface 
\[
D = \tilde{C} \cup E
\subset X.
\]
\begin{theorem} \label{thm:blowups}
    We have an exact sequence
    \[
    0 \to \pi^*(\calt_Y\langle C\rangle)(-E) \to 
    \calt_X\langle D\rangle \to \bigoplus_{p_i \in C} \calo_{E_i} \to 0.
    \]
    In particular, if $p_1,\ldots,p_k$ lie away from $C$, then
    \[
    \calt_X\langle D\rangle \simeq \pi^*(\calt_Y\langle C\rangle)(-E).
    \]
    Hence, if $p_1,\ldots,p_k$ lie away from $C$ and $C$ is free, then $D$ is free.
\end{theorem}

\begin{proof}
    Let $n=\dim(X)$. Taking differentials of $\pi$ gives the exact sequence
    \[
    0 \to \calt_X \to \pi^*(\calt_Y) \to \calt_E(-1) \to 0,
    \]
    where $\calo_E(1)$ is the taugological ample bundle of $E$ and we consider $\calt_E(-1)$ as extended by zero to $X$.
    The equisingular normal bundle $\caln'_C$, namely the subsheaf of the normal bundle $\caln_C$ locally defined by the partial derivatives of a defining equation of $C \subset Y$, fits into:
    \[
    0 \to \calt_Y\langle C \rangle \to \calt_Y \to \caln'_C \to 0.
    \]

    Assume first that the points $p_1,\ldots,p_k$ lie away from $C$ so $\tilde{C}=\pi^*(C)$. Then we have
    \[
    0 \to \calo_E(-1) \to \caln'_D \to \pi^*(\caln'_C) \to 0.
    \]

    We get an exact commutative diagram:
    \begin{equation} \label{4 rows}       
    \xymatrix@-1.5ex{
    &  & 0 \ar[d] \ar[r] & \calo_E(-1) \ar[d] \\
    0 \ar[r] &  \calt_X\langle D \rangle \ar[r]  \ar[d]& \pi^*(\calt_Y\langle C \rangle) \ar[d] \ar[r] & \calo_E^{\oplus n}\ar[r] \ar[d] & 0 \\
    0\ar[r] \ar[d] & \calt_X \ar[d] \ar[r] & \pi^*(\calt_Y) \ar[d] \ar[r] & \calt_E(-1) \ar[d] \ar[r] & 0 \\
    \calo_E(-1) \ar[r] & \caln'_D \ar[d] \ar[r] & \pi^*(\caln'_C) \ar[r] \ar[d] & 0 \\
     & 0 & 0
    }
   \end{equation}

    Now, since $\pi$ is constant on $E$, the restriction of $\pi^*(\calt_Y\langle C\rangle)$ to each component $E_i$ of $E$ is a trivial vector bundle, whose rank equals the rank $n_i \ge n$ of $\calt_Y\langle C \rangle$ at the point $p_i$, for $i=1,\ldots,k$. 
    Then, tensoring by $\pi^*(\calt_Y\langle C\rangle)$ the exact sequence
    \[
    0 \to \calo_X(-E) \to \calo_X \to \calo_E \to 0,
    \]
    we get 
    \begin{equation} \label{its a kernel}
            \pi^*(\calt_Y\langle C\rangle)(-E) \simeq \ker\left(\pi^*(\calt_Y\langle C\rangle) \to \bigoplus_{i=1}^k \calo_E^{\oplus n_i} \right).
    \end{equation}    
    For all $i=1\ldots,k$, since $n_i \ge n$, we may define a surjective map $\calo_E^{\oplus n_i} \to \calo_{E_i}^{\oplus n}$. 
    This gives rise to a surjection $\bigoplus_{i=1}^k \calo_E^{\oplus n_i} \to \calo_{E_i}^{\oplus n}$ and thus, comparing \eqref{its a kernel} the second row of \eqref{4 rows}, yields an injective map
    \[
    \theta :   \pi^*(\calt_Y\langle C\rangle)(-E) \to \calt_X\langle D \rangle 
    \]
    Note that these two sheaves are reflexive of rank $n$. We compute:
    \begin{align*}
    \dv(\theta) &= c_1(\calt_X\langle D \rangle) - c_1(\pi^*(\calt_Y\langle C\rangle)(-E)) 
     = c_1(X)-D - \left(\pi^*(c_1(Y)-C)-nE\right)= \\
    & = \pi^*(c_1(Y)) - (n-1)E -\pi^*(C) - E- \left(\pi^*(c_1(Y)-C)-nE\right)= 0.
    \end{align*}
    Therefore, the proof of Theorem \ref{thm:main} applies to show that $\theta$ is an isomorphism.
    
    Now we treat the case that one of the points $p_1,\ldots,p_k$, say $p=p_i$ lies in $C \setminus \sing(C)$. The general case, with say $j$ points in $C \setminus \sing(C)$ and $k-j$ in $X \setminus C$, is analogous. For the sake of the proof, we may even assume that $k=1$ so $E=E_i \simeq \p{n-1}$.
    Since $p$ lies in the smooth locus of $C$, we have a surjection $\caln_{p/Y}^\vee \to \caln_{p/C}^\vee$, which induces an embedding $E_{C,p} := \p{}(\caln_{p/C}^\vee) \simeq \p{n-2}\subset E=\p{}({\caln_{p/Y}^\vee})$. Observe that $D=E \cup \tilde{C}$, $E_{C,p} = E \cap \tilde{C}$ and  that the divisor $E_{C,p}$ of $E$ lies in the  $|\calo_E(1)|$, so that we have:
    \[
    0 \to \calo_E(-1) \to \calo_D \to \calo_{\tilde{C}} \to 0.
    \]
  
    Since $D\equiv \tilde{C}+E \equiv \pi^*(C)$ and $E \cap \tilde{C} \in |\calo_E(1)|$, so that $\calo_X(E)$ restricts trivially to $D$, we have $\caln_D \simeq \calo_D(D) \simeq \calo_D(\tilde{C})$. So twisting the above sequence by $\calo_X(\tilde{C})$ we get, since $\calo_{\tilde{C}}(\tilde{C})$:
    \[
    0 \to \calo_E \to \caln_D \to \caln_{\tilde{C}} \to 0.
    \]
    Also, the Jacobian subscheme of $D$ is the disjoint union of $E_p$ and the pull-back of the Jacobian subscheme of $C$ via $\pi$, more precisely, we have a commutative exact diagram:
    \[
    \xymatrix@-1.5ex{
    &&& 0 \ar[d] & \\
    &&& \pi^*(\caln'_C) \ar[d] \\
     &  &   \caln_{D} \ar^-{\simeq}[r] \ar[d] &  \pi^*(\caln_C) \ar[r] \ar[d] &   0 \\
    0 \ar[r] &  \calo_{E_{C,p}} \ar[r] &   \calt_D'  \ar[r] &  \pi^*(\calt'_C) \ar[d] \ar[r] &   0\\
    &  &  & 0}
    \]
    where $\calt_{\tilde{C}}$ and $\calt_D'$ are the singular tangent sheaves of $C$ and $D$, see \cite[Chapter 3]{sernesi:deformations}.

Since $\caln_D \simeq \pi^*(\caln_C)$, instead of \eqref{4 rows}  we get:
\begin{equation*} 
    \xymatrix@-1.5ex{
    & 0 \ar[d] & 0 \ar[d] & 0 \ar[d] \\
    0\ar[r] &  \calt_X\langle D \rangle \ar[r]  \ar[d]& \pi^*(\calt_Y\langle C \rangle) \ar[d] \ar[r] & \calo_E^{\oplus (n-1)}\ar[r] \ar[d] & 0 \\
    0\ar[r]  & \calt_X \ar[d] \ar[r] & \pi^*(\calt_Y) \ar[d]\ar[r] & \calt_E(-1) \ar[d] \ar[r] & 0 \\
    0 \ar[r] & \caln'_D \ar[d] \ar[r] & \pi^*(\caln'_C) \ar[r] \ar[d] & \calo_E \ar[r] \ar[d] & 0 \\
     & 0 & 0 & 0
    }
   \end{equation*}
As in the proof of the previous case, we get an injective map 
$\theta :   \pi^*(\calt_Y\langle C\rangle)(-E) \to \calt_X\langle D \rangle$, whose cokernel fits as kernel of a surjection $\calo_E^{\oplus n_i} \to \calo_E$. Comparing the first Chern classes, we find $n_i=n$ for all $i \in \{1,\ldots,k\}$ such that $p_i$ lies in $C$, which proves the desired exact sequence.
\end{proof}

\subsection{Free sequences on projective spaces} \label{sec:free}

In this part, we fix $X=\pn=\Proj\big(\kk[x_0,\dots,x_n]\big)$ for some field $\kk$. Set $\cala=\opn^{\oplus r}$ and $\calb=\bigoplus_{i=1}^k \opn(d_i)$, with $r>k$. By our hypothesis at the beginning of Section \ref{sec:tech}, $\ker(\alpha)$ is a reflexive sheaf of rank $r-k$. 
In addition, consider the following ingredients
\begin{enumerate}
\item \label{hyp1} a reflexive sheaf $\calf_0$ of rank $r-k$ together with a morphism $\nu:\calf_0\to\opn^{\oplus r}$ such that $\alpha\circ\nu=0$; 
\item \label{hyp2} a reflexive sheaf $\calf_2$ together with a morphism $\gamma:\calf_2\to\opn^{\oplus r}$ such that $\alpha\circ\gamma$ is a monomorphism.
\end{enumerate}
With these elements in mind, we set $\calf=\calf_0\oplus\calf_2$ and consider the morphism 
$$ \theta : \calf_0\oplus\calf_2 \longrightarrow \opn^{\oplus r} \quad,\quad
\theta:=\nu\oplus\gamma. $$
These conditions imply that $\ker(\alpha\circ\theta)\simeq\calf_0$, leading to a morphism $\theta_0:\calf_0\to\ker(\alpha)$ such that $\nu=\iota\circ\theta_0$, where $\iota:\ker(\alpha)\into\opn^{\oplus r}$ is natural the inclusion. Moreover, $\im(\alpha\circ\theta)\simeq\calf_2$.

The framework described above can be summarized in the following commutative diagram, which is to be compared with the diagram in display \eqref{thetas} with $\theta_2=\alpha\circ\gamma$:
\begin{equation}\label{diag:1}
\begin{split} \xymatrix@-2ex{ 
&  &  & 0 \ar[d] & \\
0\ar[r] & \calf_0\ar[d]^{\theta_0}\ar[r] & \calf\ar[d]^-{\theta}\ar[r] &  \calf_2\ar[d]^-{\alpha\circ\gamma}\ar[r] & 0 \\
0\ar[r] & \ker(\alpha) \ar[r]^-{\iota} & \opn^{\oplus r} \ar[r]^-{\alpha} & ~\bigoplus_{i=1}^k \opn(d_i) & 
} \end{split}
\end{equation}
leading to the following statement.

\begin{theorem} \label{thm:super-saito}
Fix the setup as above. There is a homogeneous polynomial $h\in\kk[x_0,\dots,x_n]$ such that
$$ \gcd\Big(\bigwedge\theta\Big) \cdot \gcd\Big(\bigwedge\alpha\Big) =h\cdot\gcd\Big(\bigwedge(\alpha\circ\gamma)\Big). $$
If $h\in\kk^\times$, then $\theta_0$ is an isomorphism.
\end{theorem}
\begin{proof}
Following the argument in the proof of Theorem \ref{thm:main}, we obtain
$$ \dv(\alpha) + \dv(\theta) = \dv(\alpha\circ\gamma) + \dv(\theta_0) .$$
The desired equality is obtained by setting $h:=\gcd\big(\bigwedge\theta_0\big)$ and invoking the second part of Lemma \ref{lem:c_1=gcd}. When $h\in\kk^\times$, then the first part of Lemma \ref{lem:c_1=gcd} implies that $\coker(\theta_0)$ is a torsion-free sheaf; however, $\rk(\coker(\theta_0))=\rk(\ker(\alpha))-\rk(\calf_0)=0$ by hypothesis, implying that $\coker(\theta_0)=0$, thus $\theta_0$ is an isomorphism.
\end{proof}

As an application of the previous statement, we will take $\alpha$ as the Jacobian matrix associated with a sequence $\sigma=(f_1,\dots,f_k)$ of algebraically independent homogeneous polynomials in $\kk[x_0,\dots,x_n]$; assume that ${\rm char}(\kk)$ does not divide $\deg(f_i)$ for $i=1,\dots,k$ and set $d_i:=\deg(f_i)-1$.

In \cite{faenzi-jardim-valles} the authors considered the Jacobian matrix $\nabla(\sigma)$, whose $i^{\rm th}$ line consists of the partial derivatives of the polynomial $f_i$, as a morphism
$$ \nabla(\sigma) : \opn^{\oplus n+1} \longrightarrow \bigoplus_{i=1}^k \opn(d_i) ; $$
the reflexive sheaf $\calt_\sigma:=\ker(\nabla(\sigma))$ is called the \textit{logarithmic tangent sheaf} associated with the sequence $\sigma$. Note that, if ${\rm char}(\kk)=0$, the hypothesis that the polynomials $(f_1,\dots,f_k)$ are algebraically independent implies that $\nabla(\sigma)$ has maximal rank, so that $\rk(\calt_\sigma)=n-k+1$. In positive characteristic, we just assume that $\nabla(\sigma)$ has maximal rank.

When $k=1$, so that $\sigma$ consists of a single polynomial $f$, we get back $\calt_f(1) \simeq \calt_{\p n}\langle D \rangle$ where $D=\VV(f)$. In this situation, the usual Saito criterion for the freeness of divisors in projective spaces can be recovered as a particular case of Theorem \ref{thm:super-saito}. 

Indeed, let $f \in \kk[x_0,\ldots,x_n]$ be a square-free homogeneous polynomial of degree coprime to $\mathrm{char}(\kk)$. We consider $n$ Jacobian syzygies, namely a map 
\[ \nu : \calf_0 = \bigoplus_{i=1}^n \calo_{\p n}(-e_j) \to \opn^{\oplus n+1}, \qquad
\mbox{with} \qquad \nabla(f) \circ \nu = 0. \]
Set $\calf=\calf_0 \oplus \calo_{\p n}(-1)$, let $\gamma: \calo_{\p n}(-1) \to \calo_{\p n}^{\oplus n+1}$ be the Euler matrix and define $\theta  = (\nu|\gamma): \calf \to \opn^{\oplus n+1}$.

\begin{lemma}[Saito criterion on $\p n$] \label{usual saito}
Fix the setting as above. Then there is a polynomial $h$ such that $\det(\nu)=hf$. Moreover $\calt_X\langle D\rangle \simeq \calf_0$ if and only if $h \in \kk^\times$.
\end{lemma}

\begin{proof}
In the notation of Theorem \ref{thm:super-saito}, $\calf_2:=\opn(-1)$ and $\alpha\circ\gamma=\deg(f)\cdot f$, which follows by the Euler identity. Since $\mathrm{char}(\kk)$ does not divide $\deg(f)$, $\alpha\circ\gamma$ is injective. The morphism $\theta$ is the Saito matrix: a $(n+1)\times(n+1)$ matrix whose columns are syzygies of $\nabla(f)$ plus the Euler derivation. Since $f$ is square-free, we have that $\gcd\big(\bigwedge \alpha\big)\in\kk^\times$. The equality in the statement of Theorem \ref{thm:super-saito} reduces to $\det(\theta)=h\cdot f$ for some homogeneous polynomial $h$, as expected. In addition, if $h\in\kk^\times$, then the assertion of Theorem \ref{thm:super-saito}  says that $\calt_{\pn}\langle D\rangle\simeq \bigoplus_{j=1}^n\opn(-e_j)$, that is, $D$ is free.
Finally, note that the converse claim, namely if $\calt_{\pn}\langle D\rangle$ splits as a sum of line bundles, then $\det(\theta)=\lambda f$ for some $\lambda\in\kk^\times$, follows from Proposition \ref{prop:converse}. 
\end{proof}

Returning to the case $k\ge2$, we say that an algebraically independent sequence $\sigma$ is \textit{free} if $\calt_\sigma$ splits as a sum of line bundles. We now apply Theorem \ref{thm:super-saito} to three situations involving several homogeneous polynomials, providing new examples of free sequences of homogeneous polynomials in $\kk[x_0,\dots,x_n]$.

\begin{example} \label{eg1}
Consider $\p{2k+1}=\Proj\big(\kk[x_{00},x_{01},\dots,x_{k0},x_{k1}]\big)$, and let $f_i=f_i(x_{i0},x_{i1})$ be a homogeneous, square-free polynomial of degree $d_i+1$ depending only of the variables $x_{i0}$ and $x_{i1}$. We argue that the sequence $\sigma=(f_1,\dots,f_k)$ is free, and $\calt_\sigma=\bigoplus_{i=0}^k\op{2k+1}(-d_i)$.

Indeed, the Jacobian matrix has the following shape
$$ \alpha:=\nabla(\sigma) = \left( \begin{array}{ccccccc}
\partial_0 f_0 & \partial_1 f_0 & 0 & 0 & \cdots & 0 & 0 \\
0 & 0 & \partial_0 f_1 & \partial_1 f_1 & \cdots & 0 & 0 \\
\vdots & \vdots & \vdots & \vdots &  & \vdots & \vdots \\
0 & 0 & 0 & 0 & \cdots & \partial_0 f_k & \partial_1 f_k
\end{array}\right) ~:~ \op{2k+1}^{\oplus 2k+2} \to \bigoplus_{i=0}^k \op{2k+1}(d_i)$$
where $\partial_j f_i:=\partial f_i/\partial x_{ij}$ for $j=0,1$. We then consider the morphism
$$ \nu := \left( \begin{array}{cccc} 
\partial_1 f_0 & 0 & \cdots & 0 \\
-\partial_0 f_0 & 0 & \cdots & 0 \\
0 & \partial_1 f_1 & \cdots & 0 \\
0 & -\partial_0 f_1 & \cdots & 0 \\
\vdots & \vdots &  & \vdots \\
0 & 0 & \cdots & \partial_1 f_k \\
0 & 0 & \cdots & -\partial_0 f_k
\end{array}\right) ~:~ \bigoplus_{i=0}^k\op{2k+1}(-d_i) \to \op{2k+1}^{\oplus 2k+2} $$
so that $\alpha\circ\nu=0$. 

We set $\calg:=\opn(-1)^{\oplus k+1}$ and let $\gamma$ be the morphism, 
$$ \gamma := \left( \begin{array}{cccc} 
x_{00} & 0 & \cdots & 0 \\
x_{01} & 0 & \cdots & 0 \\
0 & x_{10} & \cdots & 0 \\
0 & x_{11} & \cdots & 0 \\
\vdots & \vdots &  & \vdots \\
0 & 0 & \cdots & x_{k0} \\
0 & 0 & \cdots & x_{k1}
\end{array}\right) ~:~ \op{2k+1}(-1)^{\oplus k+1} \to \op{2k+1}^{\oplus 2k+2} $$
thus $\alpha\circ\gamma$ is a diagonal $(k+1)\times(k+1)$ matrix whose entries are
$$ \Big( ~ (d_0+1)f_0 ~~\cdots~~ (d_k+1)f_k ~ \Big) .$$

The morphism $\theta:=\nu\oplus\gamma$ is then given by a $(2k+2)\times(2k+2)$ matrix which, after re-arranging its columns, becomes block-diagonal with $2\times2$ blocks of the form
$$ \left( \begin{array}{cc} 
\partial_1 f_i & x_{i0} \\
-\partial_0 f_i & x_{i1} 
\end{array} \right). $$
It follows that 
$$ \gcd\Big(\bigwedge(\alpha\circ\gamma)\Big) = \det(\alpha\circ\gamma) = \Pi_{i=0}^k (d_i+1)f_i = \det(\theta) = \gcd\Big(\bigwedge\theta\Big) $$
Since $\gcd\big(\bigwedge \alpha\big)\in\kk^\times$ (because each $f_i$ is square-free), Theorem \ref{thm:super-saito} implies that $\calt_\sigma \simeq \bigoplus_{i=0}^k\op{2k+1}(-d_i)$, as desired.
\end{example}

\begin{example} \label{eg2}
    More generally, given a partition of $n$ as $n=n_0+\cdots+n_s$, with $n_i \ge 1$ we may take variables :
    \[
    x_{0,0},\ldots,x_{0,n_0}, \ldots,x_{i,0},\ldots,x_{i,n_i},\ldots    x_{s,0},\ldots,x_{s,n_s}, 
    \]
    and homogeneous polynomials
    \[
    \sigma=(f_{1,1},\ldots,f_{1,k_1}, \ldots,    f_{s,1},\ldots,f_{s,k_s}),
    \]
    such that, for all $i \in \{1,\ldots,s\}$, we have:
    \[
    f_{i,1},\ldots,f_{i,k_i} \in \kk[x_{i,0},\ldots,x_{i,n_i}], 
    \]
    
    If $\sigma$ is algebraically independent then also 
    $\sigma_i=f_{i,1},\ldots,f_{i,k_i}$ is algebraically independent for all $i \in \{1,\ldots,s\}$.
   Then, for all $i \in \{1,\ldots,s\}$ we may consider the sheaf $\calt_{\sigma_i}$ on the linear space $\p{n_i}$ defined by the equations $x_{j,k}=0$ for all $j \ne i$ and $k \in \{0,\ldots,n_k\}$.
    Such sheaf has a unique lift $\hat{\calt}_{\sigma_i}$ to $\pn$, which is locally free at $\p{n_i}$ if and only if $\calt_{\sigma_i}$ is free. The module associated with such sheaf has the same presentation as the module associated with $\calt_{\sigma_i}$, with only the variables $x_{i,0},\ldots,x_{i,n_i}$ showing up.
    On the other hand, the sheaf $\hat{\calt}_{\sigma_i}$ is obtained via pull-back and direct image of $\calt_{\sigma_i}$ by the blow-up diagram:
    \[
    \p{n_i} \leftarrow \hat{\p{}}^{n_i} \to \pn,
    \]
    where $\hat{\p{}}^{n_i}$ is the blow-up of $\pn$ at the linear space defined by $x_{i,0}=\cdots=x_{i,n_i}=0$.
    Then our main result implies:
    \[
    \calt_\sigma \simeq \oplus_{i=1}^s \hat \calt_{\sigma_i} 
    \]
\end{example}

\begin{example} \label{eg3}
    Let us work in characteristic zero. Let $C \subset \p{d}$ be a rational normal curve and consider its tangent developable surface $X$, namely the union of tangent projective lines to $C$.
    This is a non-normal surface of degree $2d-2$, singular along $C$. According to \cite{aprodu-farkas-papadima-raicu-weyman},
    one has $\omega_X \simeq \calo_X$ and the minimal graded free resolution of the ideal of $X$ in $\p d$ is understood.
    
    \begin{itemize}
        \item For $d=3$, an equation of $X$ is :
        \[
        \sigma=(x_0^2x_1^2-4x_0^3x_2-4x_1^3x_3+6x_0x_1x
     _2x_3-x_2^2x_3^2)
        \]
        Then $X$ is free, $\calt_\sigma \simeq \calo_{\p3}^3(-1)$ and with Saito matrix:
        \[
\begin{pmatrix}
      \vphantom{\left\{-1\right\}}x_{3}&x_{0}&2x_{1}\\
      \vphantom{\left\{-1\right\}}2x_{0}&-x_{1}&x_{2}\\
      \vphantom{\left\{-1\right\}}3x_{1}&-3x_{2}&0\\
      \vphantom{\left\{-1\right\}}0&3x_{3}&3x_{0}\\
      \end{pmatrix}
        \]
        Stacking the column vector of indeterminates to the left of the above matrix and taking the determinant gives $6 \sigma$. \\
        
        \item For $d=4$, $X$ is a complete intersection of a quadric and a cubic, we may take:
        \[
        \sigma=(f,g)=(x_2^2-2x_1x_3+2x_0x_4, 2x_2^3-6x_1x_2x_3+9x_0x_3^2+6x_1^2x_4-12
      x_0x_2x_4).
        \]
        Then we get $\calt_\sigma \simeq \calo_{\p4}^3(-1)$ and a matrix of syzygies of $\nabla(\sigma)$ is:
        \[
\begin{pmatrix}
      \vphantom{\left\{-1\right\}}2\,x_{1}&2\,x_{0}&0\\
      \vphantom{\left\{-1\right\}}3\,x_{2}&x_{1}&x_{0}\\
      \vphantom{\left\{-1\right\}}3\,x_{3}&0&x_{1}\\
      \vphantom{\left\{-1\right\}}2\,x_{4}&-x_{3}&x_{2}\\
      \vphantom{\left\{-1\right\}}0&-2\,x_{4}&x_{3}\\
      \end{pmatrix}
        \]

        Note however that $X$ is not strongly free, as one could also take
        \[
        \sigma'=(f,x_0f+g).
        \]
        as a system of minimal generators of the ideal of $X$, and check that $\calt_{\sigma'}$ has a minimal graded free resolution of the following form:
    \[
    0 
        \rightarrow 
    \calo_{\p4}(-5)     
   \rightarrow 
    \calo_{\p4}(-4)^{\oplus 5} 
    \rightarrow     \calo_{\p4}(-3)^{\oplus 5} \oplus \calo_{\p4}(-2) \oplus
        \calo_{\p4}(-1)  
   \rightarrow \calt_{\sigma'}
    \rightarrow 
    0
    \]
        
    \item For $d=5$, $X$ is a surface of degree $8$ which is the intersection of 3 quadrics in $\p5$, defining a web $\sigma$. Direct computation with Macaulay 2 tells us that the sheaf $\calt_\sigma$ is simple and fits into an exact sequence:
    \[
    0 \to \calo_{\p5}(-5)^{\oplus6} \to \calo_{\p5}(-4)^{\oplus18} \to \calo_{\p5}(-3)^{\oplus15} \to \calt_\sigma \to 0.
    \]
\item For $d=6$, $X$ is a surface of degree $10$ in $\p6$ whose ideal is generated by 6 algebraically independent quadrics. We get $\calt_\sigma \simeq \calo_{\p6}(-6)$.\\
    
    \item For $d \ge 7$, $X$ is generated by quadrics. 
    The Jacobian of a system $\sigma$ of quadric generators of the ideal of $X$ is injective. 
    Hence $\calt_\sigma=0.$
    \end{itemize}
\end{example}

\subsection{Freeness for positive characteristics} \label{sec:char>0}

In the same context as in the previous section, take a sequence $\sigma=(f_1,\dots,f_{n-1})$ of algebraically independent homogeneous polynomials in $\kk[x_0,\dots,x_n]$.  

\begin{theorem}\label{thm:positive}
If ${\rm char}(\kk)$ divides $\deg(f_i)$ for each $i=0,\dots,n-1$, then $\calt_\sigma \simeq \opn(-1)\oplus\opn(-d)$ with 
$$ d=\sum_{i=1}^{n-1}\deg(f_i) - (n-1) - \deg\Big(\gcd\big(\bigwedge\nabla(\sigma)\big)\Big) + 1. $$
\end{theorem}
\begin{proof}
Define the following:
\begin{align*}
&\cale_1=\opn^{\oplus(n+1)}, && \cale_2=\bigoplus_{i=1}^{n-1}\opn(\deg(f_i)-1), && \calf=\opn(-1), \\
&\varphi=\nabla(\sigma),&& \theta=(x_0 ~\cdots~ x_n).
\end{align*}

Since ${\rm char}(\kk)$ divides $\deg(f_i)$, we have that 
$$ \sum_{j=0}^n x_j\partial_j f_i=0, \qquad \textrm{for each $i\in \{1,\dots,n-1\}$},$$
therefore $\nabla(\sigma)\circ\theta=0$. In addition, note that $D(\theta)=\emptyset$. Therefore, we can apply Theorem \ref{thm:alternative} to conclude:
$$ \calt_\sigma \simeq \opn(-1)\oplus\opn(-d). $$
\end{proof}

\bibliographystyle{amsalpha}
\bibliography{basic_algebraic_geometry}

\newcommand{\etalchar}[1]{$^{#1}$}
\providecommand{\bysame}{\leavevmode\hbox to3em{\hrulefill}\thinspace}
\providecommand{\MR}{\relax\ifhmode\unskip\space\fi MR }
\providecommand{\MRhref}[2]{%
  \href{http://www.ams.org/mathscinet-getitem?mr=#1}{#2}
}
\providecommand{\href}[2]{#2}
\begin{thebibliography}{AFP{\etalchar{+}}19}

\bibitem[AFP{\etalchar{+}}19]{aprodu-farkas-papadima-raicu-weyman}
Marian Aprodu, Gavril Farkas, \c{S}tefan Papadima, Claudiu Raicu, and Jerzy
  Weyman, \emph{Koszul modules and {G}reen's conjecture}, Invent. Math.
  \textbf{218} (2019), no.~3, 657--720.

\bibitem[FJV21]{faenzi-jardim-valles}
Daniele Faenzi, Marcos Jardim, and Jean Vallès, \emph{Logarithmic sheaves of
  complete intersection, with ann appendix by {A}lan {M}uniz}, arXiv e-print
  mathAG/2106.14453, to appear in Ann. Sc. Norm. Super. Pisa Cl. Sci., 2021.

\bibitem[FJV24]{affine-saito}
\bysame, \emph{Saito criterion and its avatars}, Rendiconti del Circolo
  Matematico di Palermo Series 2 (2024), ArXiv ep-int math.AG/2402.08305.

\bibitem[HL97]{huybrechts-lehn:moduli}
Daniel Huybrechts and Manfred Lehn, \emph{The geometry of moduli spaces of
  sheaves}, Aspects of Mathematics, E31, Friedr. Vieweg \& Sohn, Braunschweig,
  1997.

\bibitem[OSS80]{okonek-schneider-spindler}
Christian Okonek, Michael Schneider, and Heinz Spindler, \emph{Vector bundles
  on complex projective spaces}, Progress in Mathematics, vol.~3, Birkh\"auser
  Boston, Mass., 1980.

\bibitem[PSS18]{anurag-saxena-sinhababu}
Anurag Pandey, Nitin Saxena, and Amit Sinhababu, \emph{Algebraic independence
  over positive characteristic: new criterion and applications to locally
  low-algebraic-rank circuits}, Comput. Complexity \textbf{27} (2018), no.~4,
  617--670.

\bibitem[Sai80]{saito:logarithmic}
Kyoji Saito, \emph{Theory of logarithmic differential forms and logarithmic
  vector fields}, J. Fac. Sci. Univ. Tokyo Sect. IA Math. \textbf{27} (1980),
  no.~2, 265--291.

\bibitem[Ser06]{sernesi:deformations}
Edoardo Sernesi, \emph{Deformations of algebraic schemes}, Grundlehren der
  Mathematischen Wissenschaften [Fundamental Principles of Mathematical
  Sciences], vol. 334, Springer-Verlag, Berlin, 2006.

\end{thebibliography}

\end{document}